\documentclass[]{interact}

\usepackage{epstopdf}
\usepackage{subfigure}

\usepackage{natbib}
\bibpunct[, ]{(}{)}{;}{a}{}{,}
\renewcommand\bibfont{\fontsize{10}{12}\selectfont}

\theoremstyle{plain}
\newtheorem{theorem}{Theorem}[section]
\newtheorem{lemma}[theorem]{Lemma}

\theoremstyle{definition}
\newtheorem{definition}[theorem]{Definition}
\newtheorem{example}[theorem]{Example}

\theoremstyle{remark}
\newtheorem{remark}{Remark}
\newtheorem{notation}{Notation}

\usepackage{float}

\begin{document}

	\title{Graph theoretic proofs for some results on banded inverses of $M$-matrices}
	
	\author{
		\name{S. Pratihar\textsuperscript{a}
        and K.C. Sivakumar\textsuperscript{a}\thanks{CONTACT K.C. Sivakumar. Email: kcskumar@iitm.ac.in}}
		\affil{\textsuperscript{a}Department of Mathematics, Indian Institute of 
		Technology Madras,
		Chennai, 600036, India}}

	\maketitle
	
	\begin{abstract}
This work concerns results on conditions guaranteeing that certain banded $M$-matrices have banded inverses. As a first goal, a graph theoretic characterization for an off-diagonal entry of the inverse of an $M$-matrix to be positive, is presented. This result, in turn, is used in providing alternative graph theoretic proofs of the following: (1) a characterization for a tridiagonal $M$-matrix to have a tridiagonal inverse. (2) a necessary condition for an $M$-matrix to have a pentadiagonal inverse. The results are illustrated by several numerical examples.
\end{abstract}
	

    \begin{amscode}
        05C50; 05C20; 15A09; 15B48
    \end{amscode}
	
\section{Introduction}
Let $\textbf{M}_n \mathbb{(R)}$ denote the set of all real square matrices of order $n$. For $A \in \textbf{M}_n \mathbb{(R)}$ we use the notation $A \geq 0$ to signify that the matrix $A$ has all its entries non-negative. For a vector $x$ with real coordinates, we use $x>0$ to denote that all the coordinates of $x$ are positive. The determinant of $A$ is denoted by $\det A$. Let $\alpha , \beta \subseteq \{1,2,\dotsc,n\}$, whose elements are assumed to be arranged in the ascending order. We use $A[\alpha | \beta]$ to denote the submatrix of $A$ containing the rows indexed by $\alpha$ and the columns indexed by $\beta$. $A[\alpha]:=A[\alpha|\alpha]$ denotes a principal submatrix of $A$ and $A(\alpha)$ denotes the principal submatrix $A[\alpha^{c}]$, where $\alpha^{c}$ denotes the complement of $\alpha$ in $\{1,2,\dotsc,n\}$. 

A {\it $Z$-matrix} is a real matrix whose off-diagonal entries are non-positive. Any such matrix $A$ may be represented as $A=sI-B$, where $s$ is a real number and $B \geq 0$. In particular, if $s > \rho(B)$, the spectral radius of $B$, then, we refer to $A$ as an {\it $M$-matrix}. The book \cite{Ber94} has more than fifty necessary and sufficient conditions for a $Z$-matrix to be an $M$-matrix. In what follows, we present a sample list. 

\begin{theorem}\label{invertible M matrix} Let $A$ be a $Z$-matrix. Then the following are equivalent: 
\begin{enumerate}
 \item  $A$ is an $M$-matrix.
    \item Every real eigenvalue of each principal submatrix of $A$ is positive.
    \item  All the principal minors of $A$ are positive.
    \item  $A^{-1}$ exists and $A^{-1} \geq 0$.
    \item There exists $x>0$ such that $Ax>0.$
\end{enumerate}
\end{theorem}

It follows, at once, that the diagonal entries of any $M$-matrix must be positive. In numerical examples, the fifth item above is quite useful. In particular, if the row sums of a $Z$-matrix $A$ are positive, then $A$ is an $M$-matrix. We refer the reader to the treatise \cite{Ber94}, for many applications of $M$-matrices.

We recall that a {\it tridiagonal} matrix $A = (a_{i j}) \in \textbf{M}_n \mathbb{(R)}, n \geq 3$ is a square matrix such that $a_{i j} = 0$, for $|i -j| > 1$.  $A \in \textbf{M}_n \mathbb{(R)}, n \geq 4$ is called {\it pentadiagonal}, if $a_{i j} = 0$, for $|i -j| > 2$.

In a recent work, a characterization was given for an $M$-matrix to have a tridiagonal inverse \cite[Theorem 1]{BarPen19}. The authors also considered the case of pentadiagonal matrices \cite[Theorem 2]{BarPen19}. The proofs given in op. cit. are computational in the sense that they employed the Gaussian elimination process for the inverse matrix. In this short note, we present a graph theoretic proof of both these results, which we believe are very elementary.

\section{A basic result on $M$-matrices}

\begin{definition}\label{digraph}
For a matrix $A = (a_{i j}) \in \textbf{M}_n \mathbb{(R)}$, the digraph (directed graph) $\mathcal{D}(A)$ of $A$, has $\{v_1,v_2,\dotsc, v_n\}$ as the vertex set, and whose edge set consists of those ordered pairs $(v_i,v_j), ~i \neq j$, if $a_{i j} \neq 0$.
\end{definition}

\begin{remark}
We shall be making use of a result of \cite{May89}, where the vertices are denoted by $\{1,2,\ldots, n\}$. However, we find it convenient and suggestive to use the notation for the vertices, as in Definition \ref{digraph}. 
\end{remark}

\begin{notation} 
Observe that, the word path always stands for a directed path. A set of distinct vertices $\{v_i, v_{i+1}, \dots, v_{j-1},v_j\}$ is called a path from vertex $v_i$ to vertex $v_j$, in $\mathcal{D}(A)$, if $(v_i, v_{i+1}), (v_{i+1}, v_{i+2}), \dotsc, (v_{j-1}, v_{j})$ are (directed) edges. For the sake of simplicity we use $p(v_i \to v_j)$ to denote this path. The length of $p$ will then be denoted by $l(p)$ and the set of vertices belonging to the path $p$ will be denoted by $V[p]$. The set of vertices of $\mathcal{D}(A)$ not belonging to the path $p$ will be denoted by $V(p)$. If $p$ is a path in $\mathcal{D}(A)$, we let $A[p]$ denote the corresponding product of elements of $A$. Let us make this more precise. For vertices $v_{i_1}, v_{i_k}$, set 
$$A[p(v_{i_1} \to v_{i_k})]=\begin{cases}  1 & \text{ if } k=1, \\ \prod \limits_{r=1}^{k-1}a_{i_{r}i_{r+1}}  & \text{ if } k>1. \end{cases}$$
Thus, $A[p(v_i \to v_j)]$ denotes the product of those elements of $A$, that lie along the given path $p$ of ${\cal D}(A)$, from the vertex $v_i$ to the vertex $v_j$. We set $A[p(v_i \to v_j)] = 0$, if and only if there is no path from vertex $v_i$ to vertex $v_j$.
\end{notation}

The first main result will be proved by employing the next theorem. 

\begin{theorem}\cite[Corollary 9.1]{May89}\label{entries of $A^{-1}$}
Let $A=(a_{ij}) \in \textbf{M}_n \mathbb{(R)}$ be nonsingular. Set $A^{-1} = (\tilde{a}_{i j})$. Then, we have $$ \tilde{a}_{i i} = \det A(i)/ \det A,$$ and 
\begin{equation} \label{inverse entry}\tilde{a}_{i j}= \frac{1}{\det A} \sum_{p(v_i \to v_j)} (-1)^{l(p)} A[p(v_i \to v_j)] \det A[V(p)], \;\;\; \text{ for } i \neq j, 
\end{equation} 

where the sum is taken over all paths $p$ in $\mathcal{D}(A )$ from $v_i$ to $v_j.$
\end{theorem}

Note that if $A$ is an $M$-matrix, then all its principal minors are positive, as mentioned earlier. Thus, from the formula above, it follows that the diagonal entries of its inverse are positive. Our first main result, given below, concerns the off-diagonal entries of $A^{-1}$.

\begin{theorem}\label{graph}
    Let $A$ be an $M$-matrix and set $A^{-1}:=(\tilde{a}_{ij})$. Then  
    $\tilde{a}_{ij}>0, i \neq j$ if and only if there is a path from $v_i$ to $v_j$ in $\mathcal{D}(A)$.
\end{theorem}
\begin{proof}
Since $A$ is an $M$-matrix, each of its principal minors is positive. This implies that $\det A$ as well as $\det A[V(p)]$ are positive, for any path $p$. Also, since $A[p(v_i \to v_j)]$ is a product of $l(p)$ non-positive numbers, we infer that for every path $p$,
$$(-1)^{l(p)} A[p(v_i \to v_j)]  \geq 0$$ and so  
\begin{equation}
(-1)^{l(p)} A[p(v_i \to v_j)]\det A[V(p)] \geq 0. 
\end{equation}
Now, let $\tilde{a}_{ij} >0$ for some $i,j$. Since each term in (\ref{inverse entry}) is non-negative, it follows that there is at least one path $p$ such that $$(-1)^{l(p)} A[p(v_i \to v_j)]\det A[V(p)] > 0,$$ so that $A[p(v_i \to v_j)]\neq 0,$ giving us a required path.\\
Conversely, let there be a path $p$ from $v_i$ to $v_j$ in $\mathcal{D}(A)$, so that $A[p(v_i \to v_j)] \neq 0$. Thus, $$(-1)^{l(p)} A[p(v_i \to v_j)]\det A[V(p)] > 0$$ and so from (\ref{inverse entry}), it follows that $\tilde{a}_{ij}>0$, as the right hand side is a sum of non-negative terms. This completes the proof.
\end{proof}

\begin{example} \label{example for inverse entries using graph}
The conclusion of the Theorem above is false, in general, if $A$ is not an $M$-matrix. \\
Let $$A=\begin{pmatrix}
        1 &1 &1 \\
        0 & 1&0\\
        0&1&1
\end{pmatrix},$$ so that $\mathcal{D}(A)$ is given by\\

\begin{figure}[H]
\centering 
\includegraphics{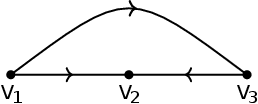}
\caption{Digraph of $A$, $\mathcal{D}(A)$.} 
\label{1st graph}
\end{figure}

Here, $A$ is not even a $Z$-matrix (and hence not an $M$-matrix). We have 
$$A^{-1}=\begin{pmatrix}
        1 &~~0 &-1\\
        0 &~~1&~~0\\
        0&-1&~~1
\end{pmatrix}.$$
Note that, while there is a path from vertex $v_1$ to vertex $v_2$ in $\mathcal{D}(A)$, the corresponding entry in $A^{-1}$ is zero ($\tilde{a}_{12} =0$).
\end{example}

\begin{example}\label{cycles}
Theorem \ref{graph} does not apply to the diagonal entries of the inverse of an $M$-matrix. The matrix $$A=\begin{pmatrix}
        ~~5 &-1 &-1 \\
        ~~0 & ~~5&~~0\\
        ~~0&-1&~~5
\end{pmatrix}$$ is an $M$-matrix. The directed graph $\mathcal{D}(A)$ is as in Figure \ref{1st graph}. Clearly, there is no path from vertex $v_i$ to itself, for any $i, ~i=1,2,3$. However, all the diagonal entries of 
$$A^{-1}=\begin{pmatrix}
        0.2 & 0.048 & 0.04 \\
        0 & 0.2&0\\
        0&0.04&0.2
\end{pmatrix},$$ 
are positive. 

\end{example}

\begin{remark}
 A real square matrix $A$ is called a P-matrix, if each principal minor of $A$ is positive. It is quite well known that any $M$-matrix is a P-matrix \citep{Ber94}. However, Theorem \ref{graph} does not extend to the larger class of P-matrices. The matrix $A$ of Example \ref{example for inverse entries using graph} serves as a counter example.
\end{remark}

\section{Tridiagonal $M$-matrices}

For a matrix $A=(a_{ij}) \in \textbf{M}_n \mathbb{(R)}$, consider the following conditions:  \begin{equation}\label{condition on tridiagonal theorem}
a_{i(i-1)} \neq 0 \implies a_{(i+1)i} =0; \hspace{.15cm} a_{(i-1)i} \neq 0 \implies a_{i(i+1)} =0, 
\end{equation} for any $i=2,3, \dotsc, n-1.$ These were proposed and studied in \cite{BarPen19}, in relation to matrices with tridiagonal inverses. Note that (\ref{condition on tridiagonal theorem}) holds if and only if the product of any {\it two} consecutive entries along the super-diagonal and the sub-diagonal, is zero.

The next result was stated in \cite{BarPen19} without a proof. We include a proof for the sake of completeness and ready reference.

\begin{lemma}\cite[Lemma 2]{BarPen19}\label{lemma for tridiagonal} 
Let $A=(a_{ij})$ be a nonsingular tridiagonal matrix with inverse $A^{-1}=(\Tilde{a}_{ij})$. If $|i-j| =1$, then $a_{ij}=0$ implies that $\Tilde{a}_{ij}=0$ and so, if $A$ satisfies (\ref{condition on tridiagonal theorem}), then $A^{-1}$ also satisfies (\ref{condition on tridiagonal theorem}).
\end{lemma}
\begin{proof} Given that $A$ is a nonsingular tridiagonal matrix. Assume that $a_{ij}=0$, for $|i-j|=1$. It follows that there is no path from vertex $v_i$ to vertex $v_j$, due to the tridiagonality of $A$. So, by (\ref{inverse entry}), we then have $\tilde{a}_{ij} =0$.\\
Thus, for $|i-j|=1$, 
\begin{equation}\label{A satisfies condition implies $A^{-1}$ also satisfy}
a _{ij}=0 \implies \tilde{a}_{ij}=0.
\end{equation}
Next, assume that $A$ satisfies (\ref{condition on tridiagonal theorem}). Then, from the above, it is clear that 
$$\tilde{a}_{ij} \neq0 \implies a_{ij} \neq 0,$$ whenever $|i-j|=1$. \\
Suppose that $\tilde{a}_{i(i-1)} \neq0.$ Then, (by (\ref{A satisfies condition implies $A^{-1}$ also satisfy})) $$a_{i(i-1)} \neq 0,$$ which in turn, implies that $$a_{(i+1)i} =0$$ (by (\ref{condition on tridiagonal theorem})). Finally, this implies that $\tilde{a}_{(i+1)i} =0  $ (by (\ref{A satisfies condition implies $A^{-1}$ also satisfy})).\\
In a similar manner, we can show that $$\tilde{a}_{(i-1)i} \neq0 \implies \tilde{a}_{i(i+1)} =0 .$$ We have shown that $A^{-1}$ satisfies (\ref{condition on tridiagonal theorem}), completing the proof. 
\end{proof}

\begin{remark}
Let us make an observation regarding matrices that satisfy (\ref{condition on tridiagonal theorem}). A matrix $A \in \textbf{M}_n \mathbb{(R)}$ is said to be \textit{reducible}, if either $A$ is a zero matrix of order $1$ or if $n \geq 2$ and there exists a permutation matrix $P$ such that $$PAP^{T} = \begin{pmatrix}
    A_{11} & 0\\
    A_{12} & A_{22}
\end{pmatrix},$$
where $A_{11}$ and $A_{22}$ are square matrices and $0$ is a zero matrix. A matrix is \textit{irreducible} if it is not reducible. It now follows that if a tridiagonal matrix $A$ satisfies (\ref{condition on tridiagonal theorem}), then $A$ is reducible.
\end{remark}

\begin{theorem}\cite[Theorem 1]{BarPen19}\label{tridiagonal}
     Let $A$ be an $M$-matrix. Then $A^{-1}$ is tridiagonal if and only if $A$ is  tridiagonal and satisfies (\ref{condition on tridiagonal theorem}).
\end{theorem}
\begin{proof}
Given that $A$ is an $M$-matrix.\\
\underline{{ Necessity}:}  Assume that $A^{-1}$ is tridiagonal. If $A$ is not tridiagonal, then there exist indices $i,j$ such that $|i-j|>1$ with $a_{ij} \neq 0$. This means that there is a path from $v_i$ to $v_j$ in $\mathcal{D}(A)$. 
Then by Theorem \ref{graph}, $$\tilde{a}_{ij} \neq 0,$$ a contradiction to the tridiagonality of $A^{-1}$. So, $A$ is tridiagonal.\\
To prove the second part, suppose that $a_{i(i-1)} \neq 0$. If $a_{(i+1)i} \neq 0$, then there is a path from $v_{i+1}$ to $v_{i-1}$. \\

\begin{figure}[H]
\centering 
\includegraphics{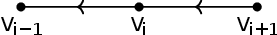}
\caption{Path from $v_{i+1}$ to $v_{i-1}$ in $\mathcal{D}(A)$.} 
\label{2nd graph}
\end{figure}

Then by Theorem \ref{graph}, $$\tilde{a}_{(i+1)(i-1)} \neq 0,$$ which contradicts that $A^{-1}$ is tridiagonal. So, our assumption is wrong, thereby proving that $$a_{i(i-1)} \neq 0 \implies a_{(i+1)i} =0.$$
The proof of the other implication $a_{(i-1)i}\neq 0 \implies a_{i(i+1)}=0$ is entirely similar and hence omitted.\\
\underline{{Sufficiency}:} Suppose that $A$ is tridiagonal and satisfies (\ref{condition on tridiagonal theorem}). 
Then the only possible directed path from the  vertex $v_i$ to the vertex $v_j$ is when there is a directed edge from $v_i$ to $v_j.$ This means that $j=i+1.$ Thus, whenever $j>i+1$, there is no path from $v_i$ to $v_j$. In an entirely similar manner, it follows that there is no path from $v_i$ to $v_j$, for $j<i-1$. Thus, there is no  path from $v_i$ to $v_j$ in $\mathcal{D}(A)$, for all $|i-j|>1$. By Theorem \ref{graph}, $\tilde{a}_{ij}=0$, for all $|i-j|>1$, proving the tridiagonality of $A^{-1}$.
\end{proof}

In view of Theorem \ref{tridiagonal} and Lemma \ref{lemma for tridiagonal}, we have the following summary.

\begin{theorem}
Let $A$ be an $M$-matrix. Then, the following are equivalent:
\begin{enumerate}
    \item $A^{-1}$ is tridiagonal.
    \item $A$ is tridiagonal and satisfies (\ref{condition on tridiagonal theorem}).
    \item $A^{-1}$ is tridiagonal and satisfies (\ref{condition on tridiagonal theorem}).
\end{enumerate}
\end{theorem}

\begin{example}
Suppose that the tridiagonal matrix given by 
$$A=\begin{pmatrix}
        a_{11} & a_{12} &&\\
        0& a_{22} & 0 &\\
         &a_{32} &a_{33} &a_{34} \\
        &&0&a_{44} 
        \end{pmatrix},$$ 
is an $M$-matrix, where $a_{ij} \neq 0.$
Note that this means that all $a_{ii}$ are positive, while for $i \neq j$, all $a_{ij}$ are negative. It is easy to verify that $A$ satisfies (\ref{condition on tridiagonal theorem}) and $A^{-1}$ is given by the tridiagonal matrix
$$A^{-1}=\begin{pmatrix}
           \frac{1}{a_{11}} & -\frac{a_{12}}{a_{11}a_{22}} & &\\
            0&~~ \frac{1}{a_{22}} &0&\\
           & -\frac{ a_{32}}{a_{22}a_{33}} & \frac{1}{a_{33}} & -\frac{a_{34}}{a_{33}a_{44}} \\
           &&0& ~~\frac{1}{a_{44}} 
        
\end{pmatrix}.$$ 
\end{example}

\begin{example}
Let $B$ be the matrix given below, obtained by changing the matrix $A$ of the previous example, in exactly one entry: 
$$B=\begin{pmatrix}
        a_{11} & a_{12} && \\
        a_{21}& a_{22} & 0 &\\
         &a_{32} &a_{33} &a_{34} \\
        &&0&a_{44}      
        \end{pmatrix}.$$
If $\gamma:=a_{11}a_{22} - a_{12}a_{21}>0,$ then $B$ is a tridiagonal $M$-matrix that does not satisfy (\ref{condition on tridiagonal theorem}). We have 
        $$B^{-1}=\frac{1}{\gamma}\begin{pmatrix}
        ~~a_{22} & -a_{12} & 0&0\\
       -a_{21} & ~~a_{11}&0&0\\
       \frac{a_{21}a_{32}}{a_{33}}  & -\frac{a_{11}a_{32}}{a_{33}} & \frac{\gamma}{a_{33}} & -\frac{\gamma a_{34}}{a_{33}a_{44}}\\
       0&0&0& ~~\frac{\gamma}{a_{44}}  
    \end{pmatrix},$$ which is not tridiagonal.
\end{example}

\section{Pentadiagonal $M$-matrices}

It is natural to examine the extent to which an analogue of Theorem \ref{tridiagonal} for the case of pentadiagonal matrices could be obtained. In this connection, a {\it necessary} condition for an $M$-matrix to have a pentadiagonal inverse was obtained in \cite[Theorem 2]{BarPen19}. The authors in op. cit. consider restrictions (like conditions (\ref{condition on tridiagonal theorem})) on the entries of a matrix $A,$ which we state below.\\
For $2 \leq i \leq n-2,$ we have the following {\it first order} conditions:
\begin{equation}\label{1st equ of first order}
    a_{i(i-1)}a_{(i+1)i} \neq 0 \implies a_{(i+2)(i+1)} =0
\end{equation}
and 
\begin{equation}\label{2nd equ of first order}
    a_{i(i+1)}a_{(i+1)(i+2)} \neq 0 \implies a_{(i-1)i} =0. 
\end{equation}
For $3 \leq i \leq n-2$, we consider, the following, the first two of which may be referred to as {\it second order} requirements:
\begin{equation}\label{1st equ of second order}
    \det A[i-2,i|i-2,i-1] \neq 0 \implies a_{(i+2)i} =0,
\end{equation}
\begin{equation}\label{2nd equ of second order}
    \det A[i,i+2|i+1,i+2] \neq 0 \implies a_{(i-2)i} =0,
\end{equation}
\begin{equation}\label{1st equ of third order}
    a_{i(i-2)} \neq 0 \implies a_{(i+2)i} = 0 \text{ and } \det A[i-1,i+1|i-1,i]=0
\end{equation}
and 
\begin{equation}\label{2nd equ of third order}
    a_{i(i+2)} \neq 0 \implies a_{(i-2)i} = 0 \text{ and } \det A[i-1,i+1|i,i+1]=0.
\end{equation}
The last two implications may be thought of as {\it third order} constraints. In a manner analogous to (\ref{condition on tridiagonal theorem}), the conditions (\ref{1st equ of first order}) and (\ref{2nd equ of first order}) imply that the product of any {\it three} consecutive terms in the super-diagonal and the sub-diagonal entries, is zero. 

\begin{theorem}\cite[Theorem 2]{BarPen19}\label{pentadiagonal}
    Let $A$ be an $M$-matrix. If $A^{-1}$ is pentadiagonal then $A$ is pentadiagonal and conditions (\ref{1st equ of first order})-(\ref{2nd equ of third order}) hold.
\end{theorem}
\begin{proof}
The matrix $A$ is an $M$-matrix with $A^{-1}$ being pentadiagonal.\\
If possible, let $A$ be not pentadiagonal. Then there exist indices $i,j$ such that $|i-j|>2$ with $a_{ij} \neq 0$ i.e., there is a path from vertex $v_i$ to vertex $v_j$ in $\mathcal{D}(A)$. Then by Theorem \ref{graph}, for this index pair $i,j$, we have $$\tilde{a}_{ij} \neq 0,$$ which contradicts that $A^{-1}$ is pentadiagonal. So our assumption is wrong, proving that $A$ is pentadiagonal.\\
Next, we verify that (\ref{1st equ of first order})-(\ref{2nd equ of third order}) hold. Firstly, suppose that $a_{i(i-1)}a_{(i+1)i} \neq 0$ (so that both $a_{i(i-1)}$ and $a_{(i+1)i}$ are nonzero). If possible, let $a_{(i+2)(i+1)} \neq 0$. \\

\begin{figure}[H]
\centering 
\includegraphics{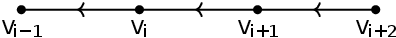}
\caption{Path from $v_{i+2}$ to $v_{i-1}$ in $\mathcal{D}(A)$.} 
\label{3rd graph}
\end{figure}

It follows that there is a path from $v_{i+2}$ to $v_{i-1}$ in ${\cal D}(A)$ (Figure \ref{3rd graph}), so that $$\tilde{a}_{(i+2)(i-1)} \neq 0.$$ This however, contradicts that $A^{-1}$ is pentadiagonal. Thus, (\ref{1st equ of first order}) holds. \\
Now, assume that $\det A[i-2,i|i-2,i-1] \neq 0$ i.e.,$$a_{(i-2)(i-2)}a_{i(i-1)} - a_{i(i-2)}a_{(i-2)(i-1)} \neq 0.$$ Since $A$ is an $M$-matrix, it then follows that the left hand side above is negative. Further, at least one of the terms is nonzero. In either case, we infer that there is a path from $v_i$ to $v_{i-1}$ (Figure \ref{4th graph}, where the path is a single edge, and Figure \ref{5th graph}). 

\begin{figure}[H]
\centering
    \subfigure[The case $a_{i(i-1)} \neq 0.$]{%
	\resizebox*{5cm}{!}{\includegraphics{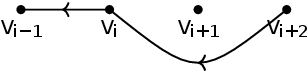}}\label{4th graph}}
 \hspace{5pt}
    \subfigure[The case $a_{i(i-2)}a_{(i-2)(i-1)} \neq 0.$]{%
	\resizebox*{5cm}{!}{\includegraphics{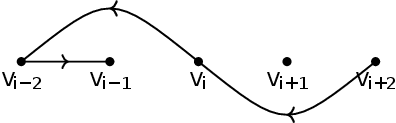}} \label{5th graph}}
\caption{Path from $v_{i+2}$ to $v_{i-1}$ in $\mathcal{D}(A).$}
\label{4th and 5th graph}
\end{figure}

Let $a_{(i+2)i} \neq 0$.
It is clear, in both the cases, that there is a path from $v_{i+2}$ to $v_{i-1}$ (Figure \ref{4th and 5th graph}), so that $$\tilde{a}_{(i+2)(i-1)} \neq 0.$$ However, this contradicts that $A^{-1}$ is pentadiagonal. So, (\ref{1st equ of second order}) holds. \\
Next, let $a_{i(i-2)} \neq 0$. If possible let,
$a_{(i+2)i} \neq 0$. Then there is a path from $v_{i+2}$ to $v_{i-2}$ in ${\cal D}(A)$ (Figure \ref{6th graph}), so that $$\tilde{a}_{(i+2)(i-2)} \neq 0.$$ This, again is a contradiction to the fact that $A^{-1}$ is pentadiagonal, allowing us to conclude that $a_{(i+2)i} =0$.\\

\begin{figure}[H]
\centering 
\includegraphics{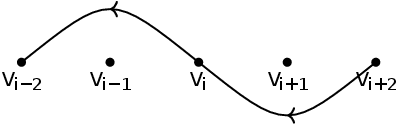}
\caption{Path from $v_{i+2}$ to $v_{i-2}$ in $\mathcal{D}(A).$} 
\label{6th graph}
\end{figure}

Next, we prove $\det A[i-1,i+1|i-1,i]=0$. If this does not hold, then 
$$a_{(i-1)(i-1)}a_{(i+1)i} - a_{(i-1)i}a_{(i+1)(i-1)} \neq 0$$ and hence at least one of them is nonzero, as argued earlier. In both the cases there is a path from $v_{i+1}$ to $v_{i}$. Also, we have $a_{i(i-2)} \neq 0$.

\begin{figure}[H]
\centering
    \subfigure[The case $a_{(i+1)i} \neq 0.$]{%
	\resizebox*{5cm}{!}{\includegraphics{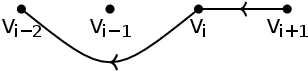}}\label{7th graph}}
\hspace{5pt}
    \subfigure[The case $a_{(i-1)i}a_{(i+1)(i-1)} \neq 0.$]{%
	\resizebox*{5cm}{!}{\includegraphics{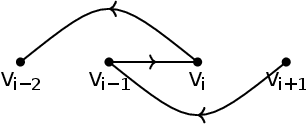}} \label{8th graph}}
\caption{Path from $v_{i+1}$ to $v_{i-2}$ in $\mathcal{D}(A).$}
\label{7th and 8th graph}
\end{figure}

Clearly, there is a path from $v_{i+1}$ to $v_{i-2}$ in $\mathcal{D}(A)$ (Figure \ref{7th and 8th graph}), and so $$\tilde{a}_{(i+1)(i-2)} \neq 0,$$ contradicting the pentadiagonality of $A^{-1}$. We have shown that (\ref{1st equ of third order}) holds.\\
By a similar argument, we can show that $A$ satisfies (\ref{2nd equ of first order}), (\ref{2nd equ of second order}) and (\ref{2nd equ of third order}), completing the proof. 
\end{proof}

We conclude this note with the following remark.

\begin{remark}
The question of whether the converse of Theorem \ref{pentadiagonal} is true, remains open. This will be investigated in future.
\end{remark}

\section*{Acknowledgements}
The authors thank the referees for the suggestions and comments that have improved the readability.

\section*{Disclosure statement}
No potential conflict of interest was reported by the authors.

\section*{Funding} 
Samapti Pratihar acknowledges financial support received through the Prime Minister's Research Fellowship (PMRF), Ministry of Education, Government of India.

\end{document}